\documentclass{amsart}

\usepackage{epsfig}
\usepackage{amscd}
\usepackage{amsmath}
\usepackage{amssymb}
\usepackage[latin1]{inputenc}

\newtheorem{theorem}{Theorem}[section]
\newtheorem{lemma}[theorem]{Lemma}
\newtheorem{corollary}[theorem]{Corollary}
\newtheorem{proposition}[theorem]{Proposition}

\newtheorem{definition}[theorem]{Definition}
\newtheorem{example}[theorem]{Example}

\newcommand{\LG}[1]{L_{#1}}
\newcommand{\bool}[1]{\mathfrak{b}(#1)}
\newcommand{\B}[2]{\mathcal{B}_{#1,#2}}
\newcommand{\TDN}{\mathcal{T}_{d,n}}
\newcommand{\M}{\mathcal{M}}
\newcommand{\T}{\mathcal{T}}
\newcommand{\Z}{\mathbb{Z}}
\newcommand{\Q}{\mathbb{Q}}
\newcommand{\R}{\mathbb{R}}
\newcommand{\RP}{\mathbb{RP}}

\newcommand{\TP}{\mathbb{TP}}
\newcommand{\tconv}{\mathrm{tconv}}
\newcommand{\leftpartial}{p}
\newcommand{\negone}[1]{\epsilon^{#1}}
\newcommand{\lefttemp}{r}
\newcommand{\rightpartial}{q}
\newcommand{\righttemp}{s}
\newcommand{\tensor}{\otimes}

\begin{document}
\author{Axel Hultman}
\address{Department of Mathematics, KTH-Royal Institute of Technology,
  SE-100 44, Stockholm, Sweden.}
\email{axel@math.kth.se}
\author{Jakob Jonsson}
\address{Department of Mathematics, KTH-Royal Institute of Technology,
  SE-100 44, Stockholm, Sweden.}
\email{jakobj@math.kth.se}

  \title[Matrices of Barvinok rank $2$]{The topology of the space of
  matrices of Barvinok rank two}

\begin{abstract}
  The Barvinok rank of a $d \times n$ matrix is the minimum number of
  points in $\mathbb{R}^d$ such that the tropical convex hull of the
  points contains all columns of the matrix. The concept originated in
  work by Barvinok and others on the travelling salesman problem.
  Our object of study is the space of real $d \times n$ matrices of
  Barvinok rank two. Let $B_{d,n}$ denote this space modulo 
  rescaling and translation.
  We show that $B_{d,n}$ is a manifold, thereby settling a
  conjecture due to Develin. In fact, $B_{d,n}$ is homeomorphic to the 
  quotient of the product of spheres $S^{d-2} \times S^{n-2}$ under
  the involution which sends each point to its antipode simultaneously in both
  components.
  In addition, using discrete Morse theory, we compute the integral
  homology of $B_{d,n}$. Assuming $d \ge n$, for odd $d$ the homology
  turns out to be 
  isomorphic to that of $S^{d-2} \times \mathbb{RP}^{n-2}$. This
  is true also for even $d$ up to degree $d-3$, but the two cases
  differ from degree $d-2$ and up. 
  The homology computation straightforwardly extends to more general
  complexes of the form 
  $(S^{d-2} \times X)/\mathbb{Z}_2$, where $X$ is a finite cell
  complex of dimension at most $d-2$ admitting a free
  $\Z_2$-action.
\end{abstract}

\maketitle

\section{Introduction}
In the {\em tropical semiring}
$(\R,\odot,\oplus)$ one defines ``multiplication'' and
``addition'' by $a\odot b = a+b$ and $a\oplus b =
\min(a,b)$, respectively. Real $n$-space $\R^n$ has the natural structure of a
semimodule over the tropical semiring so that one obtains a theory of
``tropical geometry''.\footnote{Some authors require of a semiring the
  existence of an additively neutral element, which the tropical
  semiring lacks as we have defined it. The issue is of no importance
  to us, but could be rectified by incorporating an infinity
  element.}

Given any classical geometric notion, one may try to construct reasonable
analogues in tropical geometry and study their properties. This has
been a very lively direction of research over recent years. A brief
introduction is \cite{Mikhalkin}. For more background
the reader may e.g.\ consult the extensive list of references
appearing in \cite{Litvinov}.  

Consider the classical concept of $n$ points on a line in $\R^d$. Such
point sets are characterized by the property that their convex hull is
at most one-dimensional. An equivalent characterization is that all points lie
in the convex hull of at most two of the points. In tropical geometry, there
is a natural notion of {\em tropical convex hull} which leads to obvious
analogues of the above characterizations. However, the two analogues
differ; the latter is more restrictive. Considering the
points as columns of matrices, the former 
situation deals with matrices of {\em tropical rank}
at most $2$, whereas the latter leads one to consider matrices of {\em
  Barvinok rank} at most $2$. Various tropically motivated definitions
of rank are discussed in \cite{DSS}.

The concept of Barvinok rank has found applications in optimization
theory. Motivating the nomenclature, Barvinok et al.\ \cite{BFJTWW}
showed that the maximum version of the travelling salesman problem can
be solved in polynomial time if the Barvinok rank of the distance
matrix is fixed (with $\oplus$ denoting $\max$ rather than $\min$). 

Let $\M(d,n)$ denote the set of $d\times n$ matrices with real
entries. In this paper, we are interested in the space of matrices
$M\in \M(d,n)$ with Barvinok rank $2$, corresponding to sets of $n$
marked points in 
$\R^d$ whose tropical convex hull is generated by two of the
points. The space contains some topologically less interesting
features in that it 
is invariant under rescaling and translation. Taking the quotient by the
equivalence relation generated by these operations, one obtains a space
$B_{d,n}$ which is our object of study. Develin \cite{Develin}
conjectured that $B_{d,n}$ is a manifold and that its homology ``does
not increase in complexity as $n$ gets large''. He confirmed the
conjecture for $d=3$ and computed the integral homology in a few more
accessible cases.

Our main results are summarized in the following two theorems:
\begin{theorem} \label{th:manifold}
The space $B_{d,n}$ is homeomorphic to $(S^{d-2}\times S^{n-2})/\Z_2$, 
the quotient of the product of two spheres under the involution which
sends each point to its antipode simultaneously in both components.
\end{theorem}

\begin{theorem}\label{th:homology}
The reduced integral homology groups of $B_{d,n}$ are given by
\[
\widetilde{H}_i(B_{d,n};\Z) \cong \Z^{f(i)}\oplus \Z_2^{t(i)},
\]
where
\[
f(i)=
\begin{cases}
2 & \text{if $i+2=d=n$ and $i$ is odd,}\\
1 & \text{if $i+2=n\neq d$ and $i$ is odd,}\\
1 & \text{if $i+2=d\neq n$ and $i$ is odd,}\\
1 & \text{if $i=d+n-4$ and $i$ is even,}\\
0 & \text{otherwise.}  
\end{cases}
\]
and
\[
t(i) = 
\begin{cases}
1 & \text{if $1\le i \le \min\{d,n\}-3$ and $i$ is odd,}\\
1 & \text{if $\max\{d,n\}-2\le i \le d+n-5$ and $i$ is even,}\\
0 & \text{otherwise,}  
\end{cases}
\]
\end{theorem}
When a finite group acts freely on a manifold, the quotient is again a
manifold.
In particular, therefore, the manifold
part of Develin's conjecture follows from Theorem \ref{th:manifold},
whereas the ``complexity of homology'' part is implied by Theorem
\ref{th:homology}. 

For $d \ge n$, note that the homology of 
$B_{d,n}$ coincides with that of $S^{d-2} \times (S^{n-2}/\Z_2) \cong
S^{d-2} \times \RP^{n-2}$ for odd $d$. For even $d$, the homology
groups of the two manifolds differ in higher degrees for reasons to be
explained in Section~\ref{se:homology}.

For the computations that lead to Theorem
\ref{th:homology}, we have opted to work with slightly more general
chain complexes, since this requires no additional effort. 
An advantage of this approach is that it highlights the connection
between $B_{d,n} \cong (S^{d-2}\times S^{n-2})/\Z_2$
and $S^{d-2} \times (S^{n-2}/\Z_2)$.
Specifically, for any finite cell complex $X$ with a free
$\Z_2$-action, we show that 
$(S^{d-2}\times X)/\Z_2$ and $S^{d-2} \times (X/\Z_2)$ have the same
homology whenever $\dim X \le d-2$ and $d$ is odd. For even $d$, there
is still a connection, but the situation is slightly more complicated;
see Theorem~\ref{th:Vhemi} for details.

The remainder of the paper is organized as follows. We review some
concepts from tropical geometry in the next section. In Section
\ref{se:BDN}, $B_{d,n}$ is defined and given an explicit
simplicial decomposition in terms of trees. From it, Theorem
\ref{th:manifold} is deduced. The last section is devoted to the proof
of Theorem~\ref{th:homology}. Our approach is based on 
Forman's discrete Morse theory \cite{Forman}.

\section{Tropical convexity and notions of rank}
Recall that in the tropical semiring we define $a\odot b = a+b$ and
$a\oplus b = \min(a,b)$ for $a,b\in \R$. For example,
\[
0\odot 3 \oplus (-2)\odot 3 = 1.
\]

A natural semimodule structure on $\R^n$ is provided by the ``addition''
\[
(x_1, \dots, x_n)\oplus (y_1, \dots, y_n) = (x_1\oplus y_1, \dots, x_n
\oplus y_n)
\]
and the ``multiplication by scalar''
\[
\lambda \odot (x_1, \dots, x_n) = (\lambda \odot x_1, \dots, \lambda
\odot x_n). 
\]

Following \cite{DS} we say that $S\subseteq \R^n$ is {\em tropically
  convex} if $\lambda\odot x \oplus \mu \odot y\in S$ for all $x,y\in
  S$ and $\lambda, \mu\in \R$. Note that if $S$ is tropically
  convex, then $\lambda \odot x \in S$ for all $\lambda \in \R$, $x\in
  S$. Defining {\em tropical projective space}
\[
\TP^{n-1} = \R^n/(1,\dots,1)\R,
\]
any tropically convex set in $\R^n$ is therefore uniquely determined by its
image in $\TP^{n-1}$. We obtain a convenient set of representatives
for the elements of tropical projective space by requiring the first
coordinate to be zero.

The {\em tropical convex hull} $\tconv(S)$ is the smallest
tropically convex set which contains $S\subseteq \R^n$. It coincides
with the set of finite tropical linear combinations:
\[
\tconv(S) = \left \{\bigoplus_{x\in X}\lambda_x\odot x : \lambda_x\in
  \R, \, \emptyset\neq X \subseteq S, \, |X|< \infty \right \};
\]
see \cite{DS}.

An important observation is that the tropical convex hull of two
points in $\R^n$ forms a piecewise linear curve in
$\TP^{n-1}$. This curve is the {\em tropical line segment} between the
two points. 

\begin{figure}[htb]
\begin{center}
\epsfig{height=4.5cm, file=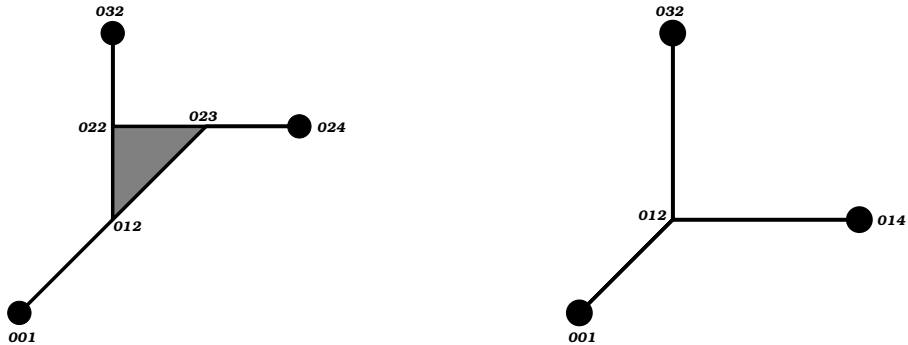}
\end{center} 
\caption{The tropical convex hulls of the point sets
  $\{(0,0,1),(0,3,2),(0,2,4)\}$ (left) and
  $\{(0,0,1),(0,3,2),(0,1,4)\}$ (right) in
  $\TP^2$.} \label{fi:hulls}
\end{figure}

In Figure \ref{fi:hulls}, the tropical convex hulls of two three-point
sets in $\TP^2$ are shown. According to the next definition, the
$3\times 3$ matrix with the three left-hand points as columns has
tropical rank $3$, whereas the points to the right form a matrix of
tropical rank $2$.

\begin{definition}[cf.\ Theorem 4.2 in \cite{DSS}]
The {\em tropical rank} of $M\in \M(d,n)$
equals one plus the dimension in $\TP^{d-1}$ of the tropical convex
hull of the columns of $M$.
\end{definition}
In particular, a matrix has tropical rank at most $2$ if and only if
the tropical convex hull of its columns is the union of the tropical
line segments between all pairs of columns. 

\begin{definition}
The {\em Barvinok rank} of $M\in \M(d,n)$ equals the
smallest number of points in $\R^d$ whose tropical convex hull
contains all columns of $M$.
\end{definition}
It is easy to see that the Barvinok rank cannot be smaller than the
tropical rank. In general, the two notions are different. For example, the
two matrices whose columns are the point sets in Figure \ref{fi:hulls}
both have Barvinok rank $3$. 

Observe that a matrix has Barvinok rank at most $2$ if and only if the
tropical convex hull of its columns is a tropical line segment.

\section{The manifold of Barvinok rank $2$ matrices}\label{se:BDN}

In this section, we shall deduce Theorem \ref{th:manifold}. To
begin with, we define the space $B_{d,n}$ which encodes the
topologically interesting part of the space of matrices of Barvinok
rank $2$. We think of $B_{d,n}$ as sitting inside $\R^{dn}$ with the
subspace topology.

Fix positive integers $d$ and $n$. Let $M\in
\M(d,n)$. As usual, we consider the columns of $M$ as a collection of $n$ marked points in
$\R^d$. Adding any real number to any row or any 
column of $M$ preserves the Barvinok rank; adding to a row merely
translates the point set, whereas adding to a column yields another
representative for the same point in $\TP^{d-1}$. Similarly,
multiplying $M$ by any $\lambda\in \R$ does not increase the Barvinok
rank (which is preserved if $\lambda \neq 0$). 

In order to get unique representatives for matrices under the
operations just described, the following definition is convenient.
\begin{definition}\label{de:representatives}
Let $B_{d,n}$ be the set of matrices $M\in
\M(d,n)$ satisfying 
\begin{itemize}
\item[(i)] The first row of $M$ is zero.
\item[(ii)] The smallest entry in every row of $M$ is zero.
\item[(iii)] As a point in $\R^{dn}$, $M$ is on the unit
sphere.
\end{itemize}
\end{definition}
\subsection{A simplicial complex of trees}

Let $P=\{p_1, \dots, p_d\}$ and $Q=\{q_1, \dots,
q_n\}$ be disjoint sets of cardinality $d$ and $n$, respectively.

We now describe an abstract simplicial complex with $B_{d,n}$ as
geometric realization. The simplices are encoded by combinatorial
trees whose leaves are marked using $P$ and $Q$ as label sets. This
model is equivalent to that given by Develin in \cite[\S~ 
3]{Develin}.\footnote{To translate from our trees to those of
  \cite{Develin}, simply replace the leaf labelled $p_i$ and its incident edge
  by a leaf ``heading off to infinity in the $i$-th coordinate
  direction'', and replace the leaf labelled $q_i$ and its incident
  edge by 
  the $i$-th marked point.} A completely analogous description in the
context of matrices of tropical rank $2$ was given by Markwig and Yu
\cite{MY}; their complex is denoted $\TDN$ below. 

Consider the set of trees $\T$ with leaf set $P\cup Q$ such that every
internal vertex (i.e.\ non-leaf) has degree at least three. We may
think of $\T$ as a simplicial complex in the following way. The vertex
set of the complex consists of all 
bipartitions of $P\cup Q$, and we identify a tree $\tau\in \T$ with the simplex
comprised of the bipartitions induced by the connected components that
arise when an internal edge (one not incident to a leaf) of $\tau$ is
removed. (It is well-known, and easy to see, that there is at most one
tree giving rise to any given set of bipartitions.) Clearly, the
simplices in the boundary of $\tau$ are those obtained by contracting
internal edges. 

Let $\TDN$ be the subcomplex of $\T$ which is induced by the bipartitions
where both parts have nonempty intersection with both $P$ and
$Q$. Our main object of study is the subcomplex of $\TDN$ which 
consists of the trees whose internal vertices form a path as induced
subgraph. Let us denote this complex by $\B{d}{n}$. 

\begin{proposition}[\S~ 3 in \cite{Develin}] \label{pr:realization}
A geometric realization of $\B{d}{n}$ is given by $B_{d,n}$.
\end{proposition}

Specifically, a matrix $M\in B_{d,n}$ is associated with a tree in $\B{d}{n}$ as
follows. Let $c_1, \dots, c_n$ denote the columns of $M$. Then,
$T=\tconv(c_1, \dots, c_n)$ is a tropical line segment in
$\TP^{d-1}$. Now construct a tree $\tau=\tau(M)$
whose internal vertex set is the union of $\{c_1, \dots, c_n\}$ and
the set of points where the curve $T$ is not smooth. Two internal
vertices $v_1\neq v_2$ are adjacent 
if and only if $\tconv(v_1, v_2)$ (which is a subset of $T$) contains
no other internal vertex. The leaf set of $\tau$ is $P\cup Q$. The
leaf $p_i$ 
is adjacent to the internal vertex which is closest to the origin
(which, by Condition (ii) of Definition \ref{de:representatives}, is
necessarily an internal vertex) among those where the $i$-th
coordinate is maximized. Finally, the leaf $q_i$ is adjacent to the
vertex $c_i$. The resulting tree (seen as an abstract, leaf-labelled tree) is
$\tau$.

\begin{example}\label{ex:example}
Consider the $6\times 5$ matrix
\[
M = \left(
\begin{array}{crrrr}
6 & 1 & 4 & 6 & 3\\
2 & -3 & -1 & 2 & -1\\
5 & -2 & 0 & 4 & 2\\
5 & -2 & 0 & 4 & 2\\
0 & -5 & -1 & 0 & -3\\
7 & -2 & 0 & 4 & 4
\end{array} 
\right).
\]
We shall see shortly that $M$ has Barvinok rank $2$. However, $M$ does
not satisfy the conditions of Definition \ref{de:representatives},
hence does not belong to $B_{6,5}$. Subtracting the first row of $M$
from each row, adding an appropriate amount to each row and
rescaling, we obtain the following matrix which represents the
equivalence class of $M$ in $B_{6,5}$:
\[
M^\prime = \frac{1}{\sqrt{97}}\left(
\begin{array}{ccccc}
0 & 0 & 0 & 0 & 0\\
1 & 1 & 0 & 1 & 1\\
3 & 1 & 0 & 2 & 3\\
3 & 1 & 0 & 2 & 3\\
0 & 0 & 1 & 0 & 0\\
4 & 1 & 0 & 2 & 5
\end{array} 
\right).
\]
The tropical convex hull of the columns of $M^\prime$ is shown in
Figure \ref{fi:example}. It is generated by the third and fifth
columns. Thus, $M^\prime$ (and $M$) has Barvinok rank $2$. The associated tree
$\tau(M^\prime)$ encoding the simplex in $\B{6}{5}$ which contains
$M^\prime$ is also displayed.
\end{example}

\begin{figure}[htb]
\begin{center}
\epsfig{height=6.5cm, file=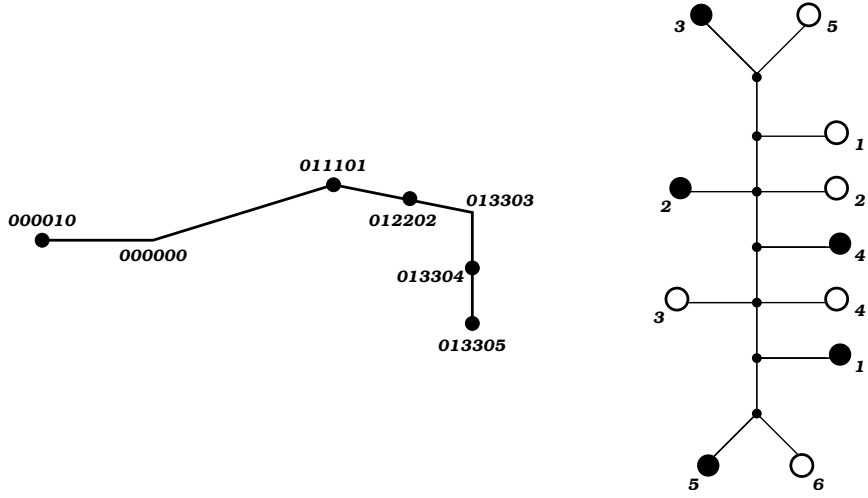}
\end{center}
\caption{(Left) The tropical convex hull of five points on a tropical line
  segment in $\TP^5$. Coordinates of the five points and the
  singular points on 
  the curve are indicated; the point configuration from Example
  \ref{ex:example} has been inflated by a factor
  $\sqrt{97}$ in order to obtain integer coordinates. (Right) The tree in $\B{6}{5}$ associated with the five points. Leaves
  represented by white dots are elements of $P$, whereas black leaves
  belong to $Q$.}\label{fi:example} 
\end{figure}

\subsection{Proof of Theorem \ref{th:manifold}}
Our next goal is to prove that $\B{d}{n}$ is homeomorphic to
$(S^{d-2}\times S^{n-2})/\Z_2$, where the generator of $\Z_2$ acts by
the antipodal map on both components. 

For a set $S$, let $\bool{S}$ denote the proper part of the Boolean
lattice on $S$. Thus, $\bool{S}$ is the poset of all proper,
nonempty subsets of $S$ ordered by inclusion. 

Recall our sets $P$ and $Q$. The chains
in $\bool{P}\times \bool{Q}$ can be thought of as compositions (i.e.\
ordered set partitions) of $P\cup Q$ such that the first and the last
block both have nonempty intersection with both $P$ and
$Q$; let us for brevity call such compositions {\em balanced}. Namely,
a chain $(S_1,T_1)< \cdots < (S_k,T_k)$ corresponds to 
the balanced composition $(C_1, \dots, C_{k+1})$, where
\[
C_i = \left(\bigcup_{j=1}^i (S_j\cup T_j)\right)\setminus  \left(\bigcup_{j=1}^{i-1} (S_j\cup T_j)\right),
\]
identifying $S_{k+1}$ and $T_{k+1}$ with $P$ and $Q$,
respectively. Under this bijection, inclusion among chains corresponds
to refinement among compositions.  

Let $\Delta(\cdot)$ denote order complex\footnote{The order complex of
a finite poset is the abstract simplicial complex whose simplices are
the totally ordered subsets.}. We have a map of simplicial
complexes $\varphi: \Delta \left(\bool{P}\times \bool{Q}\right)\to \B{d}{n}$ by
sending a balanced composition $C = (C_1, \dots, C_k)$ to the unique
tree $\varphi(C)\in \B{d}{n}$ in which $C_i$ is the set of leaves adjacent
to the $i$th internal vertex, counting along the internal path from one of
the endpoints. As an example, there are two balanced compositions that
are mapped to the tree in Figure \ref{fi:example}, namely
$(p_5q_3,p_1,p_2q_2,q_4,p_3p_4,q_1,p_6q_5)$ and its reverse
composition $(p_6q_5,q_1,p_3p_4,q_4,p_2q_2,p_1,p_5q_3)$. 

Conversely, for $\tau\in \B{d}{n}$, suppose the path from one endpoint of
the internal path to the other traverses the internal vertices in the
order $v_1, \dots, v_k$. Let $C_i\subset P\cup Q$ be the set of leaves
adjacent to $v_i$. Clearly, 
\[
\varphi^{-1}(\tau)=\{(C_1, \dots, C_k), (C_k,\dots, C_1)\}.
\]
This shows that $\varphi$ induces an isomorphism of simplicial
complexes:
\[
\Delta\left(\bool{P}\times \bool{Q}\right)/\Z_2 \cong \B{d}{n},
\]
where the generator of $\Z_2$ acts by taking complement inside
$\bool{P}$ and $\bool{Q}$ simultaneously.

Given finite posets $\Pi$ and $\Sigma$, the product space
$\Delta(\Pi)\times\Delta(\Sigma)$ has a natural cell complex structure, where
the cells are of the form $\Delta(C^\Pi)\times \Delta(C^\Sigma)$ for
chains $C^\Pi\subseteq \Pi$, $C^\Sigma\subseteq \Sigma$. It is
well known \cite[Lemma 8.9]{ES} that $\Delta(\Pi\times \Sigma)$ is a
simplicial subdivision 
of $\Delta(\Pi)\times\Delta(\Sigma)$, a cell $\Delta(C^\Pi)\times
\Delta(C^\Sigma)$ being subdivided by $\Delta(C^\Pi\times C^\Sigma)$.

In particular, $\Delta\left(\bool{P}\times \bool{Q}\right) \cong
\Delta\left(\bool{P}\right) \times
\Delta\left(\bool{Q}\right)$. Moreover, the $\Z_2$-action clearly respects the
subdivision so that we obtain
\[
\B{d}{n} \cong \left(\Delta\left(\bool{P}\right) \times
\Delta\left(\bool{Q}\right)\right)/\Z_2,
\]
where the generator of $\Z_2$ acts by taking complementary chains in
both $\bool{P}$ and $\bool{Q}$. 

Finally, it is well known that $\Delta\left(\bool{S}\right)$ is homeomorphic to
the $(|S|-2)$-sphere, and that the complement map on $\bool{S}$
corresponds to the antipodal map on the sphere. This concludes the
proof of Theorem \ref{th:manifold}.

  \section{Computing the homology}
  \label{se:homology}
  The main aim of the remainder of the paper is to compute the
  integral homology groups of $B_{d,n}$, thereby proving Theorem
  \ref{th:homology}. 
  By Theorem~\ref{th:manifold}, $B_{d,n}$ is homeomorphic to 
  $(S^{d-2}\times S^{n-2})/\Z_2$. To compute the homology of 
  this manifold, we consider the standard cell decomposition into 
  hemispheres of each of $S^{d-2}$ and $S^{n-2}$; see
  Section~\ref{hemisphere-sec} for a description. 

  It is useful, however, to work with slightly more general
  chain complexes. Thus, we shift gears and temporarily forget about the
  context of the previous sections. 

  Let $R$ be a principal ideal domain of odd or zero
  characteristic. Let
  \[
  \begin{CD}
    \mathsf{V} : \cdots @>\partial>> V_{d+1} @>\partial>> V_{d}
    @>\partial>> V_{d-1} @>\partial>> \cdots  \\
    \mathsf{W} : \cdots @>\partial>> W_{d+1} @>\partial>> W_{d}
    @>\partial>> W_{d-1} @>\partial>> \cdots 
  \end{CD}
  \]
  be chain complexes of $R$-modules equipped with a
  degree-preserving $\Z_2$-action. This means that for each of the two
  chain complexes there is a degree-preserving involutive
  automorphism $\iota$ commuting with $\partial$. 

  Consider the tensor product $\mathsf{V} \tensor \mathsf{W}$ over
  $R$; the $k$th chain group is equal to 
  \[
  \bigoplus_{i+j = k} V_i \tensor W_j,
  \]
  and the boundary map is given by
  \[
  \partial(v \tensor w) = \partial(v) \tensor w + (-1)^i v
  \tensor \partial(w)
  \]
  for $v \in V_i$ and $w \in W_j$.
  We obtain a $\Z_2$-action on $\mathsf{V} \tensor \mathsf{W}$ by
  \[
  \iota(v \tensor w) = \iota(v) \tensor \iota(w).
  \]

  For any chain complex $\mathsf{C}$ equipped with a $\Z_2$-action
  induced by the involution $\iota$, let $\mathsf{C}^+$
  be the chain complex obtained 
  by identifying an element $c$ with zero whenever $c+\iota(c) = 0$.
  Moreover, define $\mathsf{C}^-$
  to be the chain complex obtained in the similar manner 
  by identifying an element $c$ with zero whenever $c-\iota(c) = 0$.
  Our goal is to examine $(\mathsf{V} \tensor \mathsf{W})^+$.

  \subsection{The hemispherical chain complex}
  \label{hemisphere-sec}

  Let us consider the special case of interest in our calculation of
  the homology of $B_{d,n}$. Write $D = d-2$ and $N = n-2$. For $0 \le
  i \le D$, let $V_i$ be a free $R$-module generated by two 
  elements  $\sigma_i^+ = \sigma_{i}^{+1}$ and $\sigma_i^- =
  \sigma_{i}^{-1}$; set $V_i = 0$ for $i<0$ and $i>D$. We define
  \begin{equation}
    \partial(\sigma_{i}^\epsilon) = \sigma_{i-1}^\epsilon + (-1)^i
    \sigma_{i-1}^{-\epsilon}
    \label{hemiboundary-eq}
  \end{equation}
  for $\epsilon = \pm 1$.
  This means that $\mathsf{V}$ is the unreduced chain complex
  corresponding to the standard hemispherical 
  cell decomposition of the $D$-sphere.
  A $\Z_2$-action is given by mapping $\sigma_{i}^+$
  to $\sigma_{i}^{-}$ and vice versa. This corresponds to the
  antipodal action on the sphere, and 
  $\mathsf{V}^+$ consequently corresponds to the minimal cell
  decomposition of real projective $D$-space
  \cite[Ex. 2.42]{Hatcher}.
  We refer
  to $\mathsf{V}$ as the standard hemispherical chain complex over
  $R$ of degree $D$. Using Theorem~\ref{th:manifold}, we deduce the
  following.
  \begin{lemma}
    Let $\mathsf{V}$ and $\mathsf{W}$ be the standard hemispherical
    chain complexes over $R$ of degree $D = d-2$ and $N = n-2$,
    respectively. Then, the homology of $\left(\mathsf{V} \tensor
      \mathsf{W}\right)^+$ is isomorphic to the unreduced homology
    over $R$ of $B_{d,n}$.
  \end{lemma}

  \begin{lemma}
    \label{lem:proj}
    Suppose that $\mathsf{V}$ is the standard hemispherical chain
    complex over $R$ of degree $D$.
    Then,
    \[
    H_i(\mathsf{V}^+) \cong
    \left\{
    \begin{array}{ll}
      R          & \text{if $i=0$,}\\
      R/(2R)     & \text{if $1\le i < D$ and $i$ is odd,}\\
      R          & \text{if $i=D$ and $D$ is odd,}\\
      0          & \text{otherwise,}
    \end{array}
    \right.
    \]
    and 
    \[
    H_i(\mathsf{V}^-) \cong 
    \left\{
    \begin{array}{ll}
      R/(2R)     & \text{if $0 \le i < D$ and $i$ is even,}\\
      R          & \text{if $i=D$ and $D$ is even,}\\
      0          & \text{otherwise.}
    \end{array}
    \right.
    \]
  \end{lemma}
  \begin{proof}
    For each $k$, we have that $\sigma_k^+ = \sigma_k^-$ in the
    group $V_k^+$. Hence,
    \[
    \partial(\sigma_{i}^+) = \sigma_{i-1}^+ + (-1)^i
    \sigma_{i-1}^- = (1+(-1)^i)\sigma_{i-1}^+
    \]
    in $V_{i-1}^+$, which is $2\sigma_{i-1}^+$ if $i$ is even and $0$
    if $i$ is odd. Since the characteristic of $R$ is odd or zero, the
    first claim follows. 

    For the second claim, note that 
    $\sigma_k^- = - \sigma_k^+$ in the group
    $V_k^-$. Therefore,
    \[
    \partial(\sigma_{i}^+) = \sigma_{i-1}^+ + (-1)^i
    \sigma_{i-1}^- = (1-(-1)^i)\sigma_{i-1}^+
    \]
    in $V_{i-1}^-$, 
    which is $2\sigma_{i-1}^+$ if $i$ is odd and $0$ if
    $i$ is even. This proves the claim.
  \end{proof}

  \subsection{$2$ is a unit in $R$}
  \label{se:2unit}

  First we consider the case that $2$ is a unit in $R$.

  Let $\mathsf{C}$ be a chain complex of $R$-modules with an
  involution $\iota$. Then, 
  we may write each element $c$ uniquely as a sum $c = a + b$ 
  such that $a = \iota(a)$ and $b = - \iota(b)$. Namely, 
  $a = \frac{1}{2}(c + \iota(c))$
  and $b = \frac{1}{2}(c - \iota(c))$.
  This means that $\mathsf{C}^+$ can be identified with
  the subcomplex of elements $a$ satisfying $a - \iota(a) = 0$
  and $\mathsf{C}^-$ with the subcomplex of elements $b$ satisfying 
  $b + \iota(b) = 0$. Moreover, 
  we may identify $\mathsf{C}$ with the direct sum
  $\mathsf{C}^+ \oplus \mathsf{C}^-$.

  Applying the above to each of $\mathsf{V}$ and $\mathsf{W}$, we
  obtain that 
  \[
  \mathsf{V} \tensor \mathsf{W}
  = 
  \left(\mathsf{V}^+ \tensor \mathsf{W}^+\right) \oplus
  \left(\mathsf{V}^- \tensor \mathsf{W}^-\right) \oplus
  \left(\mathsf{V}^+ \tensor \mathsf{W}^-\right) \oplus
  \left(\mathsf{V}^- \tensor \mathsf{W}^+\right).
  \]
  We have that $\iota(x) = x$ if 
  \[
  x \in 
  \left(\mathsf{V}^+ \tensor \mathsf{W}^+\right) \oplus
  \left(\mathsf{V}^- \tensor \mathsf{W}^-\right)
  \]
  and 
  $\iota(x) = -x$ if 
  \[
  x \in 
  \left(\mathsf{V}^+ \tensor \mathsf{W}^-\right) \oplus
  \left(\mathsf{V}^- \tensor \mathsf{W}^+\right).
  \]
  As a consequence, 
  \[
  \left(\mathsf{V} \tensor \mathsf{W}\right)^+ 
  = 
  \left(\mathsf{V}^+ \tensor \mathsf{W}^+\right) \oplus
  \left(\mathsf{V}^- \tensor \mathsf{W}^-\right).
  \]
  Applying K\"unneth's theorem \cite[Th. 3B.5]{Hatcher}, we obtain the
  following 
  result.   
  \begin{proposition}
    \label{prop:kunneth}
    If $H_i(\mathsf{V}^+)$ and $H_i(\mathsf{V}^-)$ 
    are finitely generated free $R$-modules for each $i$, then
    \[
    H_d(\left(\mathsf{V} \tensor \mathsf{W}\right)^+)
    \cong 
    \sum_{i+j = d} 
    \left(H_i(\mathsf{V}^+) \tensor H_j(\mathsf{W}^+)\right)
    \oplus
    \left(H_i(\mathsf{V}^-) \tensor H_j(\mathsf{W}^-)\right).
    \]
  \end{proposition}

  Suppose that $\mathsf{V}$ is the standard hemispherical chain complex
  of degree $D$. By Lemma~\ref{lem:proj}, we have that 
  $H_i(\mathsf{V}^+) \cong H_i(\mathsf{V}^-) \cong 0$ 
  unless $i = 0$ or $i=D$. 
  Moreover, $H_0(\mathsf{V}^+) \cong R$ and
  $H_0(\mathsf{V}^-) \cong 0$. 
  Finally, 
  if $D$ is odd, then $H_D(\mathsf{V}^+) \cong R$ and  
  $H_D(\mathsf{V}^-) \cong 0$.
  If instead $D$ is even, then
  $H_D(\mathsf{V}^+) \cong 0$ and  
  $H_D(\mathsf{V}^-) \cong R$. The following assertion is an immediate
  consequence:

  \begin{proposition}
    \label{prop:kunneth-hemi}
    If $\mathsf{V}$ is the standard hemispherical chain complex
    of degree $D$, then
    the homology of $\left(\mathsf{V} \tensor \mathsf{W}\right)^+$ consists
    of one copy of  
    $H_j(\mathsf{W}^+)$ in degree $j$ for each $j$
    and one copy of 
    either $H_j(\mathsf{W}^+)$ or $H_j(\mathsf{W}^-)$ in degree $D+j$
    for each $j$. The former is the case if $D$ is odd, the latter
    if $D$ is even. 
\end{proposition}

Identifying $\mathsf{V}$ and $\mathsf{W}$ with the cellular chain
complexes corresponding to the hemispherical cell decompositions of
the $(d-2)$-sphere and the $(n-2)$-sphere, respectively, yields the
following corollary, which could also be deduced using transfer
methods; see e.g.\ Bredon \cite[\S~ III]{Bredon}.

\begin{corollary}
  \label{cor:freepart}
  If $2$ is invertible in the coefficient ring $R$, then the reduced
  homology of $B_{d,n}$ is given by
  \[
  \widetilde{H}_i(B_{d,n}) \cong 
  \begin{cases}
    R^2 & \text{if $i+2=d=n$ and $i$ is odd,}\\
    R & \text{if $i+2=n\neq d$ and $i$ is odd,}\\
    R & \text{if $i+2=d \neq n$ and $i$ is odd,}\\
    R & \text{if $i=d+n-4$ and $i$ is even,}\\
    0 & \text{otherwise.}
  \end{cases}
  \]
  In particular, the free part of the integral homology of $B_{d,n}$ is 
  as described in Theorem~\ref{th:homology}.
\end{corollary}
\begin{proof}
  This is an immediate consequence of Lemma~\ref{lem:proj}  
  and Proposition~\ref{prop:kunneth-hemi}. For the free part of the
  integral homology, set $R = \Q$ and apply the universal coefficient
  theorem; see e.g.\ \cite[Cor. 3A.6]{Hatcher}. 
\end{proof}

  \subsection{Discrete Morse theory on chain complexes}
  \label{se:dmt}
  To compute the homology of $\left({\sf V}
  \tensor {\sf W}\right)^+$ in the case that $2$ is not a unit in
  $R$, we will use an algebraic version \cite[\S~4.4]{thesis} of
  discrete Morse theory \cite{Forman}. 

  The general situation is that we have a chain complex 
  ${\sf C}$ of finitely generated $R$-modules $C_i$. Write $C =
  \bigoplus_i C_i$. Assume that $C$ can be written as a direct sum of
  three $R$-modules $A$, $B$, and $U$ such that $f = \alpha \circ
  \partial$ defines an isomorphism $f: B \rightarrow A$, where
  $\alpha(a+b+u) = a$ for $a \in A, b \in B, u \in U$. 
  Let $h: A \rightarrow B$ be the inverse of $f|_B$. 
  For any chain group element $x$, define $\beta(x) = h \circ
  f(x)$ and $\hat{U} = (\text{id} -\beta)(U)$. 
  Let $\hat{U}_k$ be the component of
  $\hat{U}$ in degree $k$.

  \begin{proposition}[{\cite[Th. 4.16, Cor. 4.17]{thesis}}]
    \label{prop:dmt}
    With notation as above, we have that 
    \[
    \begin{CD}
      \hat{\mathsf{U}} : 
      \cdots @>\partial>> \hat{U}_{k+1} @>\partial>> \hat{U}_{k}
      @>\partial>> \hat{U}_{k-1} @>\partial>> \cdots 
    \end{CD}
    \]
    forms a chain complex with the same homology as 
    the original chain complex $\mathsf{C}$. 
    Moreover, for each $u \in U$, the element $\beta(u)$ is the
    unique element $b \in B$ with the property that 
    $\partial(u-b) \in B+U$.
  \end{proposition}
  \begin{proof}
    For the reader's convenience, we give a proof outline. 
    Let $u \in U$. Note that 
    \[
    f(u-\beta(u))
    = f(u) - f\circ h\circ f(u) = f(u) - f(u)
    = 0;
    \]
    hence $\partial(u - \beta(u))$ is of the form $u_0 - b_0$, where 
    $u_0 \in U$ and $b_0 \in B$. Since $u_0-b_0$ is a cycle, we have
    that $f(b_0) = f(u_0)$ and hence
    \[
    u_0 - b_0 = u_0 - h \circ f(b_0)
    = u_0 - h \circ f(u_0) = u_0 - \beta(u_0),
    \]
    which implies that $\partial(\hat{U}) \subseteq \hat{U}$.

    We claim that we may write $C$ as a direct sum $\partial(B) \oplus
    B \oplus \hat{U}$. 
    Namely, let $x \in C$, and define 
    $\hat{a} = \partial \circ h \circ \alpha(x)$. We have that 
    $\alpha(\hat{a}) = \alpha(x)$, which implies that 
    $x - \hat{a} = b+u$ for some $b \in B$ and $u \in U$. Defining 
    $\hat{b} = b + \beta(u)$ and $\hat{u} = u - \beta(u)$, we obtain 
    that we may write $x = \hat{a} + \hat{b} + \hat{u}$, where
    $\hat{a} \in \partial(B)$, $\hat{b} \in B$, and $\hat{u} \in
    \hat{U}$. It is easy to show that this decomposition of $x$ is
    unique; hence we obtain the claim.
    
    Write $M = \partial(B) \oplus B$.  
    As $\partial(M) \subseteq M$, we deduce that ${\sf C}$ splits
    into the direct sum of $\hat{\sf U}$ and  
    \[
    \begin{CD}
      \mathsf{M} : 
      \cdots @>\partial>> M_{k+1} @>\partial>> M_{k}
      @>\partial>> M_{k-1} @>\partial>> \cdots 
    \end{CD}.
    \]
    The homology of the latter complex is zero, because $\partial : B
    \rightarrow \partial(B)$ is an isomorphism. As a consequence, we
    are done. The very last statement in the proposition is immediate
    from the fact that $f : B \rightarrow A$ is an isomorphism.
  \end{proof}

  For the connection to discrete Morse theory \cite{Forman}, 
  consider a matching on the set of cells in a cell complex
  such that each pair in the matching is of the form $(\sigma,\tau)$,
  where $\sigma$ is a regular codimension one face of $\tau$. Let
  $A$ be the free $R$-module generated by cells matched with larger 
  cells, let $B$ be generated by cells matched with smaller cells,
  and let $U$ be generated by unmatched cells. Then, the map $f : B
  \rightarrow A$ is an isomorphism if the matching is a Morse matching
  \cite{Chari}, and $\hat{\sf U}$ is the Morse complex associated to
  the matching. 

  \subsection{$2$ is not a unit in $R$}   \label{se:2nonunit} 
   
  In the case that $2$ is not a unit, the discussion in
  Section~\ref{se:2unit} does not apply, as the chain complex no
  longer splits in the manner described. In fact, the situation is
  considerably more complicated. For this reason, we only examine the
  special case that ${\sf V}$ is the standard hemispherical chain
  complex of degree $D$.

  We also need some assumptions on $\mathsf{W}$. Specifically, we
  assume that we may write $W_j$ as a direct product
  $\LG{j} \times \LG{j}$, where $\LG{j}$ is a finitely generated
  $R$-module for each $j \in \Z$. Moreover, we assume that
  $\iota(w_0,w_1) = (w_1,w_0)$ for each element $(w_0,w_1) \in W_j$. 
  For our main result to hold, we must assume that $L_j = 0$
  unless $0 \le j \le D$. However, we will not actually use
  this assumption until Lemma~\ref{lem:U0}.

  We make no specific assumptions on the boundary
  operator on $\mathsf{W}$, which hence is of the general form
  \[
  \partial(w_0,w_1) = (\leftpartial(w_0)+\lefttemp(w_1),
  \rightpartial(w_0) + \righttemp(w_1)), 
  \]
  where $\leftpartial,
  \rightpartial, \lefttemp, \righttemp$ are maps 
  $\LG{j} \rightarrow \LG{j-1}$ such that
  $\partial^2 = 0$  
  and $\iota \partial = \partial \iota$. 
  Since $\iota(w,0) = (0,w)$, we have that 
  \[
  (\rightpartial(w),\leftpartial(w))
  = \iota \circ \partial(w,0)
  = \partial\circ\iota(w,0)
  =   (\lefttemp(w),\righttemp(w))
  \]
  and hence that 
  \[
  \partial(w_0,w_1) = (\leftpartial(w_0)+ \rightpartial(w_1),
  \leftpartial(w_1) + \rightpartial(w_0)). 
  \]
  For $\partial^2$ to be zero, it is necessary and 
  sufficient that $\leftpartial^2+\rightpartial^2 =
  \leftpartial\rightpartial + \rightpartial\leftpartial = 0$.

  Recall that we want to examine $(\mathsf{V} \tensor
  \mathsf{W})^+$. In this chain complex, we have for each $i$ and
  $w_0, w_1 \in \LG{j}$ the identity 
  \[
  \sigma_i^- \tensor (w_0,w_1)
  = \sigma_i^+ \tensor (w_1,w_0).
  \]
  In particular, we may identify $(V_i \tensor W_j)^+$
  with $\langle \sigma_i^+\rangle \tensor W_j \cong W_j
  \mathbf{e}_{i,j}$, 
  where $\mathbf{e}_{i,j}$ is a formal variable. Writing $\negone{} =
  -1$ for compactness, also note that 
  \begin{eqnarray*}
    \partial(\sigma_i^+ \tensor (w_0,w_1))
    &=& (\sigma_{i-1}^+ + \negone{i}\sigma_{i-1}^-) \tensor (w_0,w_1)
    + \negone{i}\sigma_i^+ \tensor \partial(w_0,w_1) \\ 
    &=& \sigma_{i-1}^+ \tensor (w_0+\negone{i}w_1,w_1+\negone{i}w_0) 
    \mbox{}+ \negone{i} \sigma_i^+ \tensor
    \partial(w_0,w_1).
  \end{eqnarray*}
  Identifying $(V_i \tensor W_j)^+$ with $W_j \mathbf{e}_{i,j}$ as
  described above, we may express this as
  \[
  \partial((w_0,w_1)\mathbf{e}_{i,j}) 
  = (w_0+\negone{i}w_1,w_1+\negone{i}w_0)\mathbf{e}_{i-1,j}
  + \negone{i} \partial(w_0,w_1)\mathbf{e}_{i,j-1}.
  \]
 
  We want to use Proposition~\ref{prop:dmt} to simplify $(\mathsf{V}
  \tensor \mathsf{W})^+$. 
  For $0 \le i \le D$ and $j \in \Z$, define
  \begin{eqnarray*}
    {\bf a}_{i,j} &=& (1,0)\mathbf{e}_{i,j},\\
    {\bf b}_{i,j} &=& (0,1)\mathbf{e}_{i,j}.
  \end{eqnarray*}
  Note that $(V_i \tensor W_j)^+ \cong 
  (\LG{j} \times \LG{j})\mathbf{e}_{i,j}$ is isomorphic to the 
  direct sum of $\LG{j} {\bf a}_{i,j}$ and $\LG{j}{\bf b}_{i,j}$.
  Define 
  \begin{table}[htb]
    \caption{The table indicates for different $i$ the groups in which
      $L_j{\bf a}_{i,j}$ and $L_j{\bf b}_{i,j}$ are contained.}
    \begin{tabular}{|c||l|l|}
      \hline
      $i$ & $L_j{\bf a}_{i,j}$ & $L_j{\bf b}_{i,j}$ \\
      \hline
      $D$   &  $U^{(D)}$ &  $B$ \\
      $D-1$ &  $A$       &  $B$ \\
      $\cdots$           &  $\cdots$ &  $\cdots$ \\
      $1$   &  $A$       &  $B$ \\
      $0$   &  $A$       &  $U^{(0)}$ \\
      \hline
    \end{tabular}
    \label{tab:ABU}
  \end{table}
  \begin{eqnarray*}
    A &=& \bigoplus_{i=0}^{D-1} \bigoplus_{j}  \LG{j}{\bf a}_{i,j},
    \\  
    B &=& \bigoplus_{i=1}^D \bigoplus_{j} \LG{j}{\bf b}_{i,j}, \\
    U^{(D)} &=& \bigoplus_{j}  U_{D+j}^{(D)}, \text{where }
    U_{D+j}^{(D)} = \LG{j}{\bf a}_{D,j}, \\
    U^{(0)} &=& \bigoplus_{j}  U_j^{(0)}, \text{where }
    U_j^{(0)} = \LG{j}{\bf b}_{0,j}.
  \end{eqnarray*}
  See Table~\ref{tab:ABU} for a schematic description. 
  Write $U = U^{(0)} \oplus U^{(D)}$. 
  Note that the direct sum of $A$, $B$, and $U$ constitutes the
  full chain complex $\left({\sf V} \tensor {\sf W}\right)^+$. 

  It is clear that we obtain an isomorphism $g : B \rightarrow A$ by
  assigning $g(w{\bf b}_{i,j}) = w{\bf a}_{i-1,j}$ for each $i$ and $j$
  and each $w \in L_j$. 
  We now show that discrete Morse theory indeed yields an isomorphism
  between $B$ and $A$, though the assignment is slightly more
  complicated.
  Let $\alpha$ be the
  projection map from $A + B + U$ to $A$ as defined in
  Section~\ref{se:dmt}, and let $f = \alpha \circ \partial$.
  \begin{lemma}
    \label{lem:fiso}
    We have that $f|_B : B \rightarrow A$ defines an isomorphism.
  \end{lemma}
  \begin{proof}
    For $w \in \LG{j}$, note that 
    \begin{eqnarray*}
      \partial(w{\bf b}_{i,j})
      &=&  (\negone{i}w,w){\bf e}_{i-1,j}
      + \negone{i}(\rightpartial(w),\leftpartial(w)){\bf e}_{i,j-1}
      \\
      &=&  \negone{i}w{\bf a}_{i-1,j} + w{\bf b}_{i-1,j}
      + \negone{i}\rightpartial(w){\bf a}_{i,j-1}
      + \negone{i}\leftpartial(w){\bf b}_{i,j-1}.
    \end{eqnarray*}
    In particular, 
    \[
    f(w{\bf b}_{i,j})
    = 
    \left\{
    \begin{array}{ll}
      \negone{i}w{\bf a}_{i-1,j} & \text{if $i = D$},\\
      \negone{i}(w{\bf a}_{i-1,j} + \rightpartial(w) {\bf
      a}_{i,j-1})  &  \text{if $1\le i\le D-1$}.
    \end{array}
    \right.
    \]
    For $i=D$, the term $\rightpartial(w){\bf a}_{D,j-1}$ is
    not present, as it belongs to $U$ rather than $A$.

    Each element $x$ of degree $k$ in $B$ is of the form $x =
    \sum_{i=1}^D w_{k-i} {\bf b}_{i,k-i}$, where $w_{k-i} \in
    \LG{k-i}$.
    We may express $f(x)$ in operator matrix form as
    \[
    f(x)
    = 
    \left(
    \begin{array}{ccccccc}
      \negone{1}I & 0 & 0 & \cdots &      0 &     0 \\
      \negone{1}\rightpartial & \negone{2}I  & 0 & \cdots &      0 &     0  \\
      0 & \negone{2}\rightpartial & \negone{3}I  & \cdots &      0 &     0 \\
      \cdots     &   \cdots    & \cdots &  \cdots& \cdots& \cdots\\
      0          & 0           & 0 & \cdots
      & \negone{D-1}I &
      0 \\
      0          & 0 & 0  & \cdots &  \negone{D-1}\rightpartial & \negone{D}I
    \end{array}
    \right)
    \left(
    \begin{array}{l}
      w_{k-1}\\
      w_{k-2} \\
      w_{k-3} \\
      \cdots \\
      w_{k-D+1} \\
      w_{k-D}
    \end{array}
    \right)
    \]
    in the basis $({\bf a}_{0,k-1}, {\bf a}_{1,k-2},
    \ldots, {\bf a}_{D-3,k-D+2}, {\bf a}_{D-2,k-D+1}, {\bf
    a}_{D-1,k-D})$. 
    Now, the operator matrix is invertible; its inverse is
    \[
    \left(
    \begin{array}{llllll}
      \negone{1}I & 0 & 0 & \cdots &     0 &     0 \\
      \negone{1}\rightpartial & 
      \negone{2}I & 0 & \cdots & 0 & 0 \\
      \negone{1}\rightpartial^2 & 
      \negone{2}\rightpartial &
      \negone{3}I & \cdots
      & 0 & 0 \\ 
      \cdots     &   \cdots    & \cdots &  \cdots& \cdots& \cdots\\
      \negone{1}\rightpartial^{D-2}  & 
      \negone{2}\rightpartial^{D-3}  & 
      \negone{3}\rightpartial^{D-4}  & 
      \cdots &
      \negone{D-1}I &
      0 \\
      \negone{1}\rightpartial^{D-1} & 
      \negone{2}\rightpartial^{D-2} &
      \negone{3}\rightpartial^{D-3} & 
      \cdots & 
      \negone{D-1}\rightpartial & 
      \negone{D}I 
    \end{array}
    \right).
    \]
    Since this is true in each degree $k$, we deduce that $f|_B$ is an
    isomorphism. 
  \end{proof}

  Let $\beta$ and $\hat{\sf U}$ be defined as in
  Proposition~\ref{prop:dmt}. 
  \begin{corollary}
    \label{cor:dmt}
    We have that 
    \[
    \begin{CD}
      \hat{\mathsf{U}} : 
      \cdots @>\partial>> \hat{U}_{k+1} @>\partial>> \hat{U}_{k}
      @>\partial>> \hat{U}_{k-1} @>\partial>> \cdots 
    \end{CD}
    \]
    forms a chain complex with the same homology as 
    $(\mathsf{V} \tensor \mathsf{W})^+$. 
  \end{corollary}
  \begin{proof}
    Apply Proposition~\ref{prop:dmt} and Lemma~\ref{lem:fiso}.
  \end{proof}

  Write $\hat{u} = u - \beta(u)$, $\hat{U}_k^{(D)} = (\text{id} -
  \beta)(U_k^{(D)})$, 
  and $\hat{U}_k^{(0)} = (\text{id} - \beta)(U_k^{(0)})$.
  Moreover, define ${\bf \hat{a}}_{D,j} = {\bf a}_{D,j} +
  \negone{D+1}{\bf b}_{D,j} = (1,\negone{D+1}){\bf e}_{D,j}$.
  \begin{lemma}
    \label{lem:Un}
    Consider an element in $U^{(D)}$ of the form
    $u = w {\bf a}_{D,j}$, where $w \in \LG{j}$.
    Then, we have that $\hat{u} 
    = w {\bf \hat{a}}_{D,j}$ and
    \[
    \partial(w {\bf \hat{a}}_{D,j}) =
    \negone{D}(\leftpartial(w)+\negone{D+1}\rightpartial(w)){\bf
    \hat{a}}_{D,j-1}.
    \]
    In particular, the groups $\hat{U}_j^{(D)}$ constitute a
    subcomplex $\hat{\sf U}^{(D)}$ of $\hat{\sf U}$, and 
    \[
    H_*(\hat{\sf U}^{(D)}) \cong
    \left\{
    \begin{array}{ll}
      H_{*-D}({\sf W}^+) & \text{if $D$ is odd},\\
      H_{*-D}({\sf W}^-) & \text{if $D$ is even}.
    \end{array}
    \right.
    \]
  \end{lemma}
  \begin{proof}
    The formula for $\partial(w {\bf \hat{a}}_{D,j})$ is 
    just a straightforward computation. By Proposition~\ref{prop:dmt}, 
    it follows that $\hat{u} = w {\bf \hat{a}}_{D,j}$.

    For the last statement, we may identify $\mathsf{W}^+$ with the
    chain complex with chain 
    groups $\LG{j}$ and with the boundary map given by
    \[
    \partial(w) = \leftpartial(w)+\rightpartial(w).
    \]
    Similarly, we may identify $\mathsf{W}^-$ with the chain complex
    with chain groups $\LG{j}$ and with the boundary map given by
    \[
    \partial(w) = \leftpartial(w)-\rightpartial(w).
    \]
    Up to a shift in degree by $D$, the chain groups and the boundary 
    map of $\hat{\sf U}^{(D)}$ are isomorphic to those of either
    $\mathsf{W}^+$ or $\mathsf{W}^-$, depending on the parity of $D$.
    As a consequence, we obtain the statement.
  \end{proof}

  For $a \in A$, $b \in B$, and $u \in U$, write $\gamma(a+b+u) = u$.
  Recall the assumption that $L_j = 0$ unless $0 \le j \le D$. 
  \begin{lemma}
    \label{lem:U0}
    Let $0 \le k \le D$ and let $u = w{\bf
    b}_{0,k}$, where $w \in \LG{k}$. 
    Then,
    \begin{equation}
      \hat{u} = \varphi(w{\bf b}_{0,k}) := \sum_{i=0}^D 
      \rightpartial^i(w){\bf b}_{i,k-i}
      \label{eq:hatu0}
    \end{equation}
    and
    \begin{equation}
      \partial(\varphi(w{\bf b}_{0,k}))
      =
      \varphi((\leftpartial(w)+\rightpartial(w)){\bf
        b}_{0,k-1}).
      \label{eq:hatu0b}
    \end{equation}
    In particular, the groups $\hat{U}_j^{(0)}$ constitute a
    subcomplex $\hat{\sf U}^{(0)}$ of $\hat{\sf U}$, and 
    \[
    H_*(\hat{\sf U}^{(0)}) \cong
      H_{*}({\sf W}^+).
    \]
  \end{lemma}
  \begin{proof}
    The boundary of the right-hand side of 
    (\ref{eq:hatu0}) equals
    \begin{eqnarray*}
      & & \sum_{i=1}^D 
      \rightpartial^i(w)(\negone{i}{\bf a}_{i-1,k-i}+{\bf b}_{i-1,k-i})
      \\
      &&  \mbox{} +
      \sum_{i=0}^{D}
      \negone{i}\rightpartial^{i+1}(w){\bf a}_{i,k-i-1} 
      + 
      \sum_{i=0}^{D} 
      \negone{i}
      \leftpartial\rightpartial^i(w) {\bf b}_{i,k-i-1} 
      \\
      &=& 
      \sum_{i=0}^{D-1} 
      \negone{i}
      (\leftpartial\rightpartial^i(w)
      +\negone{i}\rightpartial^{i+1}(w)) {\bf b}_{i,k-i-1} 
      \\
      && + \mbox{} \negone{D} \rightpartial^{D+1}(w){\bf a}_{D,k-D-1}
      + \negone{D} \leftpartial\rightpartial^D(w) {\bf b}_{D,k-D-1}
      \\
      &=& 
      \sum_{i=0}^{D-1} 
      \negone{i}
      (\leftpartial\rightpartial^i(w)
      +\negone{i}\rightpartial^{i+1}(w)) {\bf b}_{i,k-i-1}
      \\
      &=& 
      \sum_{i=0}^{D-1} 
      \rightpartial^i(\leftpartial(w)+\rightpartial(w)) {\bf
      b}_{i,k-i-1}
      = \varphi((\leftpartial(w)+\rightpartial(w)){\bf b}_{0,k-1}).
    \end{eqnarray*}
    The second equality is because $k\le D$ and hence
    $\leftpartial\rightpartial^D(w) = \rightpartial^{D+1}(w) = 0$.
    The third equality follows from repeated application of the
    identity $\leftpartial\rightpartial = -
    \rightpartial\leftpartial$. 
    This yields (\ref{eq:hatu0b}). Since the element in the right-hand
    side of (\ref{eq:hatu0b}) is an element in $B+U^{(0)}$, 
    Proposition~\ref{prop:dmt} yields (\ref{eq:hatu0}).
    
    For the final claim in the lemma, note that
    the chain groups and the boundary map of $\hat{U}^{(0)}$ are
    isomorphic to those of $\hat{\sf W}^+$.
  \end{proof}

  Without the assumption that $\LG{j} = 0$ unless $0 \le j \le D$, 
  (\ref{eq:hatu0b}) would not necessarily be true, and
  the groups $\hat{U}_j^{(0)}$ might not constitute a subcomplex of
  $\hat{\sf U}$. 
  Namely, the term $\rightpartial^{D+1}(w){\bf
    a}_{D,k-D-1} \in U^{(D)}$ in the expansion of $\partial(\hat{u})$
  is nonzero if $\rightpartial^{D+1}(w)$ is nonzero.
  
  \begin{corollary}
    \label{cor:Usplit}
    $\hat{\mathsf{U}}$ splits into two complexes
    \[
    \begin{CD}
      \hat{\mathsf{U}}^{(D)} : 
      \hat{U}^{(D)}_{2D} @>\partial>>
      \hat{U}^{(D)}_{2D-1} @>\partial>> 
      \cdots @> \partial>>
      \hat{U}^{(D)}_{D+1} @>\partial>> 
      \hat{U}^{(D)}_{D},\\
      \hat{\mathsf{U}}^{(0)} : 
      \hat{U}^{(0)}_{D} @>\partial>>
      \hat{U}^{(0)}_{D-1} @>\partial>> 
      \cdots @> \partial>>
      \hat{U}^{(0)}_{1} @>\partial>>
      \hat{U}^{(0)}_{0}.
    \end{CD}
    \]
  \end{corollary}
  \begin{proof}
    This is a an immediate consequence of Lemmas~\ref{lem:Un}
    and \ref{lem:U0}. 
  \end{proof}

  Applying Lemmas~\ref{lem:Un} and \ref{lem:U0} and using
  Corollaries~\ref{cor:dmt} and \ref{cor:Usplit}, we obtain 
  a description of the homology of $\left(\mathsf{V} \tensor
  \mathsf{W}\right)^+$ in terms of $\mathsf{W}^+$
  and $\mathsf{W}^-$.
  \begin{theorem}
    \label{th:Vhemi}
    If $D$ is odd, then
    \begin{eqnarray*}
    H_i(\left(\mathsf{V} \tensor \mathsf{W}\right)^+)
    &\cong& 
    H_i(\mathsf{W}^+) \oplus H_{i-D}(\mathsf{W}^+)\\
    &\cong& 
    \left\{
    \begin{array}{ll}
      H_i(\mathsf{W}^+)
      & \text{if $0 \le i \le D-1$}, \\
      H_D(\mathsf{W}^+)
      \oplus 
      H_0(\mathsf{W}^+)
      & \text{if $i = D$}, \\
      H_{i-D}(\mathsf{W}^+)
      & \text{if $D+1 \le i \le 2D$},\\
      0
      & \text{otherwise.}
    \end{array}
    \right.
    \end{eqnarray*}
    If $D$ is even, then
    \begin{eqnarray*}
      H_i(\left(\mathsf{V} \tensor \mathsf{W}\right)^+)
      &\cong& 
      H_i(\mathsf{W}^+) \oplus H_{i-D}(\mathsf{W}^-)\\
      &\cong& 
      \left\{
      \begin{array}{ll}
        H_i(\mathsf{W}^+)
        & \text{if $0 \le i \le D-1$}, \\
        H_D(\mathsf{W}^+)
        \oplus 
        H_0(\mathsf{W}^-)
        & \text{if $i = D$}, \\
        H_{i-D}(\mathsf{W}^-)
        & \text{if $D+1 \le i \le 2D$},\\
        0
        & \text{otherwise.}
      \end{array}
      \right.
    \end{eqnarray*}
  \end{theorem}

  In the case of $B_{d,n}$, we already know the free part of the
  homology by Corollary~\ref{cor:freepart}; hence we may focus on the
  torsion part.
  \begin{corollary}
    \label{cor:torpart}
    Let $d \ge n$.
    Then, the torsion part
    $T_*(B_{d,n};\Z)$ 
    of
    $H_*(B_{d,n};\Z)$ is an
    elementary $2$-group satisfying
    \[
    T_i(B_{d,n};\Z)
    \cong 
    \left\{
    \begin{array}{ll}
      \Z_2
      & \text{if $1 \le i \le n-3$ and $i$ is odd}, \\
      \Z_2 & \text{if $d-2 \le i \le d+n-5$ and $i$ is even}, \\
      0 & \text{otherwise}.
      \end{array}
    \right.
    \]
  \end{corollary}
  \begin{proof}
    This is an immediate consequence of Lemma~\ref{lem:proj}
    and Theorem~\ref{th:Vhemi}; let 
    $\mathsf{W}$ be the standard hemispherical 
    cell decomposition of the $(n-2)$-sphere.
  \end{proof}

  Combining Corollaries~\ref{cor:freepart} and \ref{cor:torpart}, 
  we obtain Theorem~\ref{th:homology}.


\begin{thebibliography}{99}
  \bibitem{BFJTWW} A.\ Barvinok, S.\ P.\ Fekete, D.\ S.\ Johnson, A.\
    Tamir, G.\ J.\ Woeginger, and R.\ Woodroofe, The geometric maximum
    traveling salesman problem,  {\em J.\ ACM} {\bf 50}
    (2003), 641--664. 
  \bibitem{Bredon} G.\ E.\ Bredon, {\em Introduction to compact transformation
      groups}, Academic Press, New York, 1972.  
  \bibitem{Chari} M.\ K.\ Chari, On discrete Morse functions and
    combinatorial decompositions, {\em Discrete Math.} {\bf 217}
    (2000), No. 1-3, 101--113.
  \bibitem{Develin} M.\ Develin, The moduli space of $n$ tropically
  collinear points in $\R^d$, {\em Collect.\ Math.\ }{\bf 56} (2005), 1--19.
  \bibitem{DSS} M.\ Develin, F.\ Santos, and B.\ Sturmfels, On the
    rank of a tropical matrix, J.\ E.\ Goodman, J.\ Pach and E.\ Welzl,
    Editors, {\em Discrete and Computational Geometry} vol. 52,
    MSRI Publications, Cambridge University Press, 2005, 213--242.
  \bibitem{DS} M.\ Develin, B.\ Sturmfels, Tropical convexity, {\em
      Doc.\ Math.\ }{\bf 9} (2004), 1--27.
  \bibitem{ES} S.\ Eilenberg and N.\ Steenrod, {\em Foundations of
  Algebraic Topology}, Princeton University Press, 1952.
  \bibitem{Forman} R.\ Forman, Morse theory for cell complexes, {\em
  Adv. in Math.} {\bf 134} (1998), 90--145. 
  \bibitem{Hatcher}
    A.\ Hatcher, {\em Algebraic Topology}, Cambridge University Press, 2002.
  \bibitem{thesis} J.\ Jonsson, {\em Simplicial Complexes of Graphs},
    Lecture Notes in Mathematics 1928, Springer, 2008.
  \bibitem{Litvinov} G.\ L.\ Litvinov, The Maslov dequantization,
    idempotent and tropical mathematics: a very brief introduction, In
    {\em Idempotent mathematics and mathematical physics}, Contemp.\
    Math.\ {\bf 377}, Amer.\ Math.\ Soc., Providence, RI, 2005, 1--17.
  \bibitem{MY} H.\ Markwig, J.\ Yu, The space of tropically
  collinear points is shellable, {\em
  Collect.\ Math.} {\bf 60} (2009), 63--77.
  \bibitem{Mikhalkin} G.\ Mikhalkin, Tropical geometry and its
    applications, {\em Proceedings of the International Congress of
    Mathematicians}, Madrid 2006, 827--852.
 
  \end{thebibliography}
\end{document}